\documentclass[draft,12pt]{article}
\usepackage{amsfonts}
\usepackage{amssymb}
\usepackage{amsmath}
\usepackage{amsthm}
\usepackage{mathrsfs}
\usepackage{amstext}
\usepackage{graphicx}
\usepackage{euscript}

\def\thebibliograph#1#2{\section*{{\normalsize \bf #2}}\list
	{[\arabic{enumi}]}{\settowidth\labelwidth{[#1]}\leftmargin\labelwidth
		\advance\leftmargin\labelsep
		\usecounter{enumi}}
	\def\newblock{\hskip .11em plus .33em minus -.07em}
	\SFoppy
	\sfcode`\.=1000\relax}

\newcommand{\sch}[1]{{\mathcal S}_{{#1}}}
\newcommand{\temp}[1]{{\mathcal S}'_{{#1}}}
\newcommand{\poly}[1]{{\mathcal P}_{{#1}}}

\newcommand{\euclid}{{\mathbb{R}}^n}
\newcommand{\N}{{\mathbb{N}}}
\newcommand{\cs}{{\mathcal{S}}}

\newcommand{\z}{{\mathbb{Z}}}
\newcommand{\R}{{\mathbb{R}^n}}
\newcommand{\re}{{\mathbb{R}}}

\newcommand{\besovn}[3]{\dot{B}_{{#2},{#3}}^{#1}({\mathbb R}^n)}

\newcommand{\lizorn}[3]{\dot{F}_{{#2},{#3}}^{#1}({\mathbb R}^n)}
\newtheorem{thm}{Theorem}[section]
 \newtheorem{cor}[thm]{Corollary}
 \newtheorem{lem}[thm]{Lemma}
 \newtheorem{prop}[thm]{Proposition}
 \theoremstyle{definition}
 \newtheorem{defn}[thm]{Definition}
 \theoremstyle{remark}
 \newtheorem{rem}[thm]{Remark}
 
 \numberwithin{equation}{section}

\newcommand{\be}{\begin{equation}}
\newcommand{\ee}{\end{equation}}

 \title{Szasz's theorem and its generalizations}

 \author{G\'erard Bourdaud} 
   \date{\today}
 
\begin{document}
 \maketitle

 \begin{abstract}
		We establish the most general Szasz type estimates for homogeneous Besov and Lizorkin-Triebel spaces, and their realizations.	\end{abstract}


 \section{Introduction}

What we call a {\em Szasz type  theorem} is the following estimate :
\begin{equation}\label{szasz}  \left(       \int_{\euclid }     |\xi|^{\theta p} |\mathcal{F}(f)(\xi)|^p\,\mathrm{d}\xi                            \right)^{1/p}      
   \leq c\, \|f\|_{\dot A^{s}_{r,q}(\R)}\,,               \end{equation}
where $\mathcal{F}$ is the  Fourier transformation, and $c$ depends only on the fixed parameters $s,p,q,r,n, \theta$. Here  $\dot A^{s}_{r,q}(\R)$ denotes the  homogeneous Besov space $\besovn{s}{r}{q}$ or the Lizor\-kin-Triebel space $\lizorn{s}{r}{q}$, and the parameters satisfy
 \begin{equation}\label{param} s\in \mathbb{R}\,, \quad p,q,r\in ]0,\infty]\,,\end{equation}
 see  Section \ref{def} for the detailed definitions.
  By an easy homogeneity argument, the value of
 $\theta$ is necessarily \begin{equation}\label{goodt} \theta=s+ n-\frac np -\frac nr\,.\end{equation}
This number $\theta$ will be referred as the {\em Szasz exponent} associated with $(s,p,q,r,n)$.
A quick reading of the estimate
 (\ref{szasz}) might lead to the following statement : {\it for all $f\in\dot A^{s}_{p,q}(\R)$,  $\mathcal{F}(f)$ is a locally integrable function  on $\euclid$ satisfying  \em{(\ref{szasz})}. } But this is trivially inexact : if $f$ is  a nonzero polynomial, the r.h.s. of (\ref{szasz}) is $0$, while $\mathcal{F}(f)$ is a nonzero distribution supported by $\{0\}$, which does not belong to $L_{1,loc}(\euclid)$.
 
 To obtain a correct formulation, it is necessary to deal with a {\em modified Fourier transformation}.  Let us recall that $\dot A^{s}_{r,q}(\R)$ is a subspace of  $ \mathcal{S}'_\infty(\euclid)$, the space of tempered distributions modulo polynomials. Then we have the following statement, whose easy proof is left to the reader :
 \begin{prop}\label{fourier+} There exists a unique one-to-one continuous linear mapping \[\dot{\mathcal F}: \mathcal{S}'_\infty(\euclid)\rightarrow
 \mathcal{D}'(\euclid \setminus \{0\})\] s.t., for all $f\in   \mathcal{S}'(\euclid)$,
 $ \dot{\mathcal F}([f]_\infty)$ is the restriction of $\mathcal{F}(f)$ to $\euclid \setminus \{0\}$, where $ [f]_\infty$ denotes the equivalence class of $f$ modulo polynomials.
 \end{prop}
 With the help of $\dot{\mathcal F}$, we can give a more precise formulation of the estimate (\ref{szasz}), namely : 
 \begin{defn} \label{conj}The system  $(s,p,q,r,n)$ satisfies the property (w-SB) (resp. (w-SF)) if there exists $c=c(s,p,q,r,n)>0$ s.t., for all $u\in\dot A^{s}_{r,q}(\R)$, the distribution  $\dot{\mathcal F}(u)$ is a locally integrable function on $\euclid\setminus \{0\}$ s.t.
 \begin{equation}\label{szaszbis}  \left(       \int_{\R \setminus \{0\}}       |\xi|^{\theta p} |\dot{\mathcal F}(u)(\xi)|^p\,\mathrm{d}\xi                            \right)^{1/p}      
   \leq c\, \|u\|_{\dot A^{s}_{r,q}(\R)}\, ,           \end{equation}
where $\theta$ is given by (\ref{goodt}), and $A$ stands for $B$ or $F$, respectively.
   \end{defn}
    We refer to both (w-SB) and (w-SF) as the ``weak Szasz properties".
     The property (w-SB) was proved classically in a number of cases.
  Some of them appeared in the book of J.~Peetre \cite{Pe}. They concern the case $\theta =0$ :
  \begin{itemize}
  \item $r=2$, $0<p\leq 2$, $q=p$, see thm.~4, p.~119, called ``Szasz theorem''. The original Szasz's result concerns periodic Besov spaces in dimension $n=1$, and the particular case $p=1$ is due to Bernstein, see thm.~3, p.~119.
  \item $1\leq r \leq 2$, $p=q=1$ (see (7'), p.~120).
  \item $0< r\leq 1$, $p=q=\infty$, see cor.~3, p.~251, first stated in (3), p.~116, for $r=1$.
  \end{itemize}
 
 Then B.~Jawerth \cite[thm.~3.1]{Ja} proved the property (w-SF) in case 
  \[ s=0\,,\quad \theta = n\left(1- \frac2p\right)\,,\quad 1<r=p<2\,.\]
The case $p=q=r=2$, $\theta =s=0$ is just Plancherel. Jawerth's theorem was proved formerly in case $q=2$, under the name ``Hardy inequality'' (Paley, Hardy and Littlewood in the periodic case, Fefferman and Bj\"ork for $p=1$, Peetre for $0<p<1$).

Peetre stated his Szasz type theorems in the following form :
\[ \mathcal{F}\,:\, \dot{B}^s_{r,q} \quad \rightarrow \quad L_p\,,\]
without taking care of polynomials. He was aware of the difficulty, but he preferred not giving too much details. He was thinking that any clever reader would interpret his statements correctly, i.e. according  to Definition \ref{conj}.

In the present paper, we first prove the Szasz property
in the most general possible case. Then, under supplementary conditions on parameters, we  will improve it,
 by making use of the classical notion of {\em realization} for homogeneous function spaces :
\begin{defn}\label{conj2} The system $(s,p,q,r,n)$ satisfies the property (s-SB) (resp. (s-SF)) if there exist a realization 
$\sigma: \dot{A}^{s}_{r,q}(\R)\rightarrow \mathcal{S}'(\R)$ and $c=c(s,p,q,r,n)>0$ s.t.
\begin{equation}\label{realinl1} \mathcal{F}(\sigma (\dot A^{s}_{r,q}(\R) )) \subset L_{1,loc}(\euclid)\end{equation}
and s.t. the estimate (\ref{szasz}) is satisfied for
 every $f\in \sigma (\dot A^{s}_{r,q}(\R) )$, where $\theta$ is  given by (\ref{goodt}),  and $A$ stands for $B$ or $F$, respectively.
 Both properties are referred as the ``strong Szasz properties".
 \end{defn}
 \begin{rem}\label{wtos}   The strong Szasz property implies the  weak one. Conversely, under condition (\ref{realinl1}), 
 the weak Szasz property implies the strong one.
\end{rem}
We will  prove the strong Szasz property
in the most general possible case, and observe that some systems $(s,p,q,r,n)$ satisfy the weak Szasz property, but not the strong one.

\section{Preliminaries}\label{def}
\subsection{ Notation}
	$\N$ denotes the set of natural numbers, including $0$.  All distribution spaces in this work are defined on  $\R$, except otherwise specified.  We set  $r':=r/(r-1)$ if $1<r\le\infty$ and $r':=\infty$ if $0<r\leq 1$. The symbol $\hookrightarrow$ indicates a continuous embedding. We denote by $\sch{}$  the Schwartz class  and by $\temp{}$  its  topological dual, the space of tempered distributions, endowed with the $\ast$-weak topology. For $0<p\le\infty$, $\|\cdot\|_p$ denotes the $L_p$ quasi-norm w.r.t. Lebesgue measure. 
	For $a\in \R$, the translation operator $\tau_a$ acts on functions according to the formula
	$(\tau_a f)(x):= f(x-a)$ for all $x\in \R$. For $\lambda>0$, the dilation operator $h_\lambda$ acts on functions according to the formula
	$(h_\lambda f)(x):= f(x/\lambda)$ for all $x\in \R$.  The Fourier transform of $f\in \temp{}$ is denoted by $\widehat{f}$. We denote by $\poly{\infty}$ the set of all polynomials on $\R$. We denote by $\sch{\infty}$  the set of all $\varphi\in\sch{}$ such that $\langle u,\varphi \rangle=0$  for all $ u\in\poly{\infty}$, and by  $\temp{\infty}$ its topological dual endowed with the $\ast$-weak topology. For all $f\in\temp{}$, $[f]_{\infty}$ denotes the equivalence class of $f$ modulo $\poly{\infty}$. The mapping which takes any $[f]_{\infty}$ to the restriction of $f$ to $\sch{\infty}$ turns out to be an isomorphism from $\temp{}/ \poly{\infty}$ onto $\temp{\infty}$.
	The constants  $c,c_1,\ldots$, are strictly positive, and depend only on the fixed parameters $n,s,p,q,r$ ; their values may change from a line to another.

	\subsection{Homogeneous Besov and Lizorkin-Triebel  spaces}\label{Besovspaces}

	The usual definition of $\dot A^{s}_{r,q}$ is given via the Littlewood-Paley decomposition, that we briefly recall. We consider a $C^{\infty }$  function $\gamma$, supported by $ 1/2\leq |\xi|\leq 3/2 $,  s.t.
	 \[\sum_{j\in\z}\gamma(2^{j}\xi)=1 \, ,\quad \forall \xi\in \R\setminus\{0\}\,,\]
		and define the operators   
	$Q_j$ $(j\in\z)$ by   $\widehat{Q_{j}f}:=\gamma(2^{-j}(\cdot)) \widehat{f}$. Then for all $f\in  \temp{\infty}$, it holds 
			$f  = \sum_{j\in {\mathbb Z}} Q_jf $ in  $ \temp{\infty}$.
			\begin{defn}Let $s\in\re$ and $0<r,q\leq\infty$ with $r<\infty$ in $F$-case.
	\begin{itemize}
			\item  [ {\rm (i)} ]  The homogeneous Besov space $\dot B^{s}_{r,q}$ is the set of $f\in\temp{\infty}$ s.t \[\|f\|_{\dot B^{s}_{r,q}}:=\,
			\bigl( \sum_{j\in{\z}}2^{jsq}\| {Q}_{j}f\|_{r}^{q} \bigr)^{1/q}<\infty\,.\]
			\item  [{\rm (ii)}] The homogeneous  Lizorkin-Triebel space $ \dot F^{s}_{r,q}$ is the set of $f\in \temp{\infty}$ s.t.
			\[\| f\|_{\dot F^{s}_{r,q}}:=\bigl\|\bigl( \sum_{j\in\z}2^{jsq}| {Q}_{j}f|^{q}\bigr) ^{1/q}\bigr\|_{r}<\infty\,.\]
		\end{itemize}	\end{defn}
	We recall the Nikol'skij type estimates, see \cite[prop.~4]{BMS_10}, \cite[prop.~3.4]{Mou10}, \cite[props.~2.15,~2.17]{Mou_15}, \cite[prop.~2.3.2/1]{RS_96} and \cite{YA} for the proofs.
	\begin{prop}\label{lizoryam}Let $0<a<b$. For all sequence
		let $(u_j)_{j\in {\mathbb Z}}$ in $\temp{}$ s.t.
		\begin{itemize}
			\item  $\widehat{u_j}$ is supported by the annulus
			$a2^j \leq |\xi|\leq b2^j$, for all $j\in \z$,
			\item$M := \big(\sum_{j\in \mathbb {Z}}
			( 2^{js} \|u_j\|_r)^q\big)^{1/q}< \infty$ in the $B$-case ,
			\item $M:=  \big \|\big (\sum_{j\in \mathbb {Z}}
			(2^{js}| u_j|)^q\big )^{1/q}\big \|_r< \infty$ in the $F$-case,
		\end{itemize}
		 the series $  \sum_{j\in \mathbb{ Z}} u_j $
		converges in $\temp{\infty}$ and
		$ \| \sum_{j\in \mathbb{ Z}} u_j \|_{\dot A^{s}_{r,q}}
		\leq cM$,
		where the constant $c$ depends only on $n,s,r,q, a,b$.
		\end{prop}
		\begin{prop}\label{LPdecompbis} For all
	$f\in \dot{A}^{s}_{r,q}$, there exists a set of polynomials $w_j$, $j\in \z$, such that the series
	$\sum_{j\in {\mathbb Z}}( Q_jf - w_j)$ is  convergent in $ \temp{}$. The sum of that series is a tempered distribution $g$ s.t.
	$[g]_\infty=f$.
	\end{prop}
	
	\begin{proof} In case of Besov space with $r=q=\infty$, see e.g. \cite[rem.~4.9]{Bou_13}. The general case follows by embedding $
	\dot{A}^{s}_{r,q} \hookrightarrow \dot{B}^{s-(n/p)}_{\infty,\infty}$.	\end{proof}

\section{General Szasz's theorem}

\subsection{Statements of the results}

\begin{thm}\label{thszasz} Let $s,p,q,r$ s.t. {\em (\ref{param})}.
\begin{itemize}
\item The system $(s,p,q,r,n)$ satisfies (w-SB) iff 
\begin{equation}\label{condB} 0<r\leq 2 \quad \mathrm{and} \quad 0<q\leq p\leq r'\,.\end{equation}

\item The system $(s,p,q,r,n)$ satisfies (w-SF) iff 
\begin{equation}\label{condL}0<r\leq 2 \quad \mathrm{and} \quad
\left(\, r\leq  p< r' \quad \mathrm{ or}\quad  q\leq  p= r'\,\right)\,.\end{equation}

\end{itemize}
\end{thm}

The above theorem has a counterpart for the usual (inhomogeneous) Besov and Lizorkin-Triebel spaces :

\begin{thm}\label{thszaszinhom} Let $s,p,q,r$ s.t. {\em (\ref{param})}, and $\theta$ the associated Szasz exponent.
Let us consider the following property: {\em $(\cal{Z})$\,There exists $c>0$ s.t., for all $f\in A^{s}_{r,q}$, $\mathcal{F}(f)$ is a locally integrable function on $\R$ satisfying the estimate
\[  \left(       \int_{\euclid }  (1+   |\xi|)^{\theta p} |\mathcal{F}(f)(\xi)|^p\,\mathrm{d}\xi                            \right)^{1/p}      
   \leq c\, \|f\|_{A^{s}_{r,q}}\,,\]}
   where $A^{s}_{r,q}$ stands for $B^{s}_{r,q}$ (resp. $F^{s}_{r,q}$).
 Then $(\cal{Z})$
holds iff (\ref{condB}) (resp. (\,\ref{condL}\,)).
\end{thm}

The above property $(\cal{Z})$ has been classically established in particular cases, see e.g. \cite[p.55]{ST}.

\subsection{Proof}

We limit ourselves to Theorem \ref{thszasz}. The same arguments work for Theorem \ref{thszaszinhom}, up to minor modifications.
 
 {\em Step 1.} We first prove that the various conditions on $r,p,q$ imply the weak Szasz property.
 
 {\em Substep 1.1.}  Let us assume $1\leq r\leq 2$, $0<p\leq r'$ and $p<\infty$.
 We introduce the annulus
 \[U_j:= \{ \xi\in \R\,:\, 2^j\leq |\xi|\leq 2^{j+1}\,\}\,, j\in \z\,.\]
 Let $u\in \dot{B}^s_{r,p}$. By the Haussdorff-Young theorem, it holds
\[ \|\widehat{Q_ju}\|_{r'}\leq c\, \|Q_ju\|_r\,.\]
Thanks to condition $p\leq r'$, we can apply H\"older's inequality and deduce
  \be\label{HY}     \int_{U_j} \left| \widehat{Q_ju}(\xi)\right|^p\,\mathrm{d}\xi
   \leq c\, 2^{jn(1-(p/r'))}    \|Q_ju\|^p_r\,.   \ee            
     By definition of $Q_j$, the function $\widehat{Q_ju}$ is supported by the annulus
                \[ 2^{j-1}\leq |\xi|\leq 3.2^{j-1}  \,.\]
                Thus the function
   \[v(\xi):= \sum_{j\in \z}       \widehat{Q_ju}(\xi)\]
is well defined and locally integrable on    $\R\setminus \{0\}$.
                We {\em claim} that
  \be\label{sz}  \left(       \int_{\euclid \setminus \{0\}}     |\xi|^{\theta p} |v(\xi)|^p\,\mathrm{d}\xi                            \right)^{1/p}      
   \leq c\, \|u\|_{{\dot B}^{s}_{r,p}}\, .           \ee
We observe that, for $\xi\in U_j$, it holds $\widehat{Q_ku}(\xi)=0$ for $k<j$ or $k>j+1$.
Hence, by (\ref{HY}) and by definition of $\theta$ (see (\ref{goodt})), we obtain
 \[\int_{\euclid \setminus \{0\}}     |\xi|^{\theta p} |v(\xi)|^p\,\mathrm{d}\xi 
=\sum_{j\in \z}\int_{U_j}     |\xi|^{\theta p} |v(\xi)|^p\,\mathrm{d}\xi\]
      \[\leq c_1\,  \sum_{j\in \z}   2^{j\theta p} \int_{U_j} |\widehat{Q_ju}(\xi)|^p\mathrm{d}\xi\leq
  c_2\, \sum_{j\in \z} 2^{sp} \|Q_ju\|_r^p\,.\]
    This ends up the proof of the {\em claim}.  By Proposition \ref{LPdecompbis}, there exists a convergent series $\sum_{j\in \z} f_j$ in $\mathcal{S}'$ s.t.
    $f_j- Q_ju$ is a polynomial for each $j\in \z$. By setting
    $f:= \sum_{j\in \z} f_j$, we obtain a tempered distribution s.t. $[f]_\infty = u$.
 The restriction of $\widehat{f}$ to $\R\setminus \{0\}$ coincide with $v$. Thus the estimate (\ref{sz}) is precisely (\ref{szaszbis}) for the function space $\dot{B}^s_{r,p}$.  We conclude that $(s,p,p,r,n)$ has the property (w-SB).
 \begin{rem}\label{p=infty}The above proof works as well in case $r=1$ and $p=\infty$, up to minor changes.
  \end{rem}
  {\em Substep 1.2.} Let us assume $0< r<1$.
 By a classical embedding, see e.g. \cite[thm.~2.1]{Ja}, it holds
 \[ \dot{B}^s_{r,p}\quad \hookrightarrow \quad \dot{B}^{s + n-(n/r)}_{1,p}\,.\]
 Both systems $(s,p,p,r,n)$ and $(s + n-(n/r),p,p,1,n)$ have the same Szasz exponent.
 By Substep 1.1 (with $r=1$), we conclude that $(s,p,p,r,n)$ has the property (w-SB).
 
 {\em Substep 1.3.} Under the condition : $0<r\leq 2$ and $q\leq p\leq r'$, we have the embedding
  $\dot{B}^s_{r,q}\, \hookrightarrow \, \dot{B}^{s}_{r,p}$,
  with the same Szasz exponent. By the preceding substeps, the system $(s,p,q,r,n)$ has the property (w-SB).
  
  {\em Substep 1.4.} For Lizorkin-Triebel spaces, we argue similarly, by using embeddings into Besov spaces, without change of Szasz exponent.
We first consider the embedding
  $\dot{F}^s_{r,q}\, \hookrightarrow \,\dot{B}^{s}_{r,\max (r,q)}$,
   see \cite[(1.2)]{Ja}. We conclude that $(s,p,q,r,n)$ has the property (w-SF) if 
   \[ 0<r\leq 2\,, \quad \max(r,q)\leq p\leq r'\,.\]
   
  {\em Substep 1.5.} Now we assume that $0<r<2$ and $r\leq p<r'$. Then we choose $r_1>r$, sufficiently near from $r$ to have
 \[ 0<r<r_1\leq 2\,,\quad r\leq p\leq r_1'\,.\]
 By setting
 \[s_1:=s+\frac{n}{r_1}- \frac{n}{r}\,,\]
we obtain  the embedding
 $\dot{F}^s_{r,q}\,\hookrightarrow \, \dot{B}^{s_1}_{r_1,r}$,
   see \cite[thm.~2.1]{Ja}.
We conclude that $(s,p,q,r,n)$ has the property (w-SF).\\
   
   {\em Step 2.} Now we prove that the Szasz property implies the  various conditions on $r,p,q$.
   
      {\em Substep 2.1.}  To prove the necessity of condition $r\leq 2$, we use the following statement :
   
   \begin{lem}\label{singmeas} For all $2<r\leq \infty$, there exist a compact subset $K$ of $\R$, with Lebesgue measure equal to $0$, s.t. $0\notin K$, and a nonzero Borel measure $\mu$ on $\R$, supported by $K$, s.t. ${\mathcal F}^{-1}\mu\in L_r(\R)$.
   \end{lem}

\begin{proof}  Let us fix a number $\beta$ s.t.
\[ \frac{1}{r} < \beta < \frac{1}{2}\,.\]
Consider first the case $n=1$. According to Kaufmann's theorem \cite{RK}, see also \cite[thm.~9A2]{TW}, there exist a compact subset $C$ of $\re$, with Lebesgue measure equal to $0$, and a nonzero Borel measure $\mu$ on $\re$, supported by $C$, s.t. the function $g:={\mathcal F}^{-1}\mu$ satisfies the estimate
\begin{equation}\label{estmu} |g(x)| \leq c\, (1+|x|)^{-\beta}\,,\end{equation}
hence $g$ belongs to $L_r(\re)$. By a translation of $C$ and $\mu$, if necessary, we have also $0\notin C$. 
In higher dimension $n$, we define $\mu_n:= \mu \otimes \cdots \otimes \mu$ ($n$ factors). Then $\mu_n$ is a Borel measure supported by $K:=C^n$, and
$f:={\mathcal F}^{-1}\mu_n$ is given by $f(x_1,\ldots,x_n)= g(x_1)\cdots g(x_n)$. Then $f\in L_r(\R)$ follows by (\ref{estmu}).
\end{proof}
Let us keep the same notation as in the above Lemma and proof.
   Since $K$ is a compact subset  of $\R\setminus \{0\}$, it holds
   \[ f= \sum_{j=N}^M Q_jf\]
   for some $N,M\in {\mathbb Z}$. Hence $f$ belongs as well to $\dot{A}^s_{r,q}$ whatever be $s,q$.
   Since the Lebesgue measure of $K$ is equal to $0$,
  $\widehat{f}$ is not locally integrable on $\R\setminus \{0\}$. Hence the Szasz property cannot hold.
   
   {\em Substep 2.2.}  Assume $r'<p\leq \infty$. Let $\varphi\in \sch{}$ s.t. $\widehat{\varphi}$ is supported by the ball $|\xi|\leq 1/2$.
    Let us set $f_k(x):= {\rm e}^{2^kix_1} \varphi (x)$, $k=1,2,\ldots$, and
    \[ f:= \sum_{k=1}^\infty k 2^{-k\theta} f_k\,.\]
    Since $\widehat{f_k}$ is supported by the annulus $C_k:= \{\xi\,:\, \frac34  2^k\leq|\xi| \leq\frac54 2^k\}$, we can apply Proposition \ref{lizoryam} and conclude that
    \[ \|f\|_{ \dot{B}^s_{r,q}}\leq c\, \left(  \sum_{k=1}^\infty  ( 2^{ks} \|k 2^{-k\theta} f_k\|_r)^q\right)^{1/q}
    =c \|\varphi\|_r\, \left(  \sum_{k=1}^\infty  ( k 2^{k(s-\theta)})^q\right)^{1/q}\,.\]
    Here we have
    \[ s-\theta = \frac{n}{p}-\frac{n}{r'}<0\,,\]
hence $f\in \dot{B}^s_{r,q}$. The series which defines $f$ converges as well in $\cs'$ and
  \[ \widehat{f}= \sum_{k=1}^\infty k 2^{-k\theta} \,\tau_{2^ke_1}\widehat{\varphi}\]
  in $\cs'$, where $e_1:=(1,0,\ldots,0)$. For every $k$, we have
\[ \left(       \int_{\R \setminus \{0\}}     |\xi|^{\theta p} |\widehat{f}(\xi)|^p\,\mathrm{d}\xi                            \right)^{1/p} 
\geq  c\, k \left(       \int_{C_k}     |\widehat{\varphi}(\xi - 2^ke_1)|^p\,\mathrm{d}\xi                            \right)^{1/p} 
= c\,k \|\widehat{\varphi}\|_p\,.\]
Thus the property (w-SB) does not hold.

 {\em Substep 2.3.}  Assume $p<q$. Let $\psi\in \sch{}$ s.t. $\widehat{\psi}$ is supported by the annulus $C_0$.
    Let us set 
    \[ f:= \sum_{k=1}^\infty k^{-1/p} 2^{k((n/r)-s)} h_{2^{-k}}\psi\,.\]
    The above series converges in $\mathcal{S}'$, and,
  by Proposition \ref{lizoryam}, $ \|[f]_\infty\|_{ \dot{B}^s_{r,q}}$ is less than
    \[ c\, \left(  \sum_{k=1}^\infty  ( 2^{ks} k^{-1/p} 2^{k((n/r)-s)}\| h_{2^{-k}}\psi\|_r)^q\right)^{1/q}=
c \|\psi\|_r\, \left(  \sum_{k=1}^\infty   k^{-q/p}\right)^{1/q}\,,\]
hence $[f]_\infty\in  \dot{B}^s_{r,q}$.
By definition of $\theta$, it holds
 \[ \widehat{f}(\xi)= \sum_{k=1}^\infty k^{-1/p} 2^{k((n/p)-\theta)} \widehat{\psi}(2^{-k}\xi)\,.\]
Hence $ \int_{\euclid \setminus \{0\}}     |\xi|^{\theta p} | \widehat{f}(\xi)|^p\,\mathrm{d}\xi $ is greater than
 \[                          
\sum_{k=1}^\infty \int_{C_{-k} } \left| k^{-1/p} 2^{k((n/p)-\theta)} \widehat{\psi}(2^{-k}\xi)\right|^p\, |\xi|^{\theta p} \,\mathrm{d}\xi 
\geq c \|\widehat{\psi}  \|^p_p \sum_{k=1}^\infty k^{-1}\,.\]                       
Thus the property (w-SB) does not hold.

{\em Substep 2.4.} Assume $p<r$. Let us take $r_1,s_1$ s.t. 
\[ p<r_1<r\,\quad s_1-\frac{n}{r_1}=s-\frac{n}{r}.\]
According to \cite[(1.3) and thm.~2.1 (ii)]{Ja}, we have the embedding
\begin{equation}\label{embedBF} \dot{B}^{s_1}_{r_1,r_1} = \dot{F}^{s_1}_{r_1,r_1}\hookrightarrow \dot{F}^{s}_{r,q}\end{equation}
with the same Szasz exponent. By condition $p<r_1$ and Substep 2.3,  the system $(s_1,p,r_1,r_1,n)$ has not the property  (w-SB). By (\ref{embedBF}), the system $(s,p,q,r,n)$ has not the property  (w-SF).

{\em Substep 2.5.} Assume $p>r'$. Let us take $r_1,s_1$ s.t. 
\[ p>r'_1>r'\,\quad s_1-\frac{n}{r_1}=s-\frac{n}{r}.\]
Then the embedding (\ref{embedBF}) holds again,
without change of Szasz exponent.
By condition $p>r'_1$ and Substep 2.2 the system $(s_1,p,r_1,r_1,n)$ has not the property  (w-SB). Hence the system $(s,p,q,r,n)$ has not the property  (w-SF).

{\em Substep 2.6.} Assume $p=r'<q$. Let us take the same functions $f_k$ as in Substep 2.2, and define
 \[ f:= \sum_{k=1}^\infty k^{-1/p} 2^{-k\theta} f_k\,.\]
 Here we have $s=\theta$. Then $\|f\|^r_{ \dot{F}^s_{r,q}}$ is less than
        \[ c\, \int_{\R} \left(  \sum_{k=1}^\infty  ( k^{-1/p} | f_k(x)|)^q\right)^{r/q} \mathrm{d}x
         = c\, \int_{\R} \left(  \sum_{k=1}^\infty   k^{-q/p} \right)^{r/q} \, | \varphi (x)|^r \mathrm{d}x\,, \]
        hence $f\in \dot{F}^s_{r,q}$. It holds
         \[      \int_{\R \setminus \{0\}}     |\xi|^{\theta p} |\widehat{f}(\xi)|^p\,\mathrm{d}\xi                           
\geq  c\,\|\widehat{\varphi}\|^p_p\sum_{k=1}^\infty      k^{-1}\,.\]
     We conclude that the property (w-SF) does not hold.

\section{ Szasz's theorem for the realized function spaces}\label{real}

\subsection{Generalities on realizations}

The proof of Theorem \ref{szasz} relies upon the choice of a specific representative for each member of $\dot A^{s}_{r,q}$, with the help of 
Proposition \ref{LPdecompbis}. This choice of representative may be organized in a coherent way, through the notion of realization, that we recall here.
\begin{defn}\label{reali} Let $E$ be a quasi-Banach space continuously embedded into $\mathcal{S}'_\infty$.
A {\em realization} of $E$ is a continuous linear mapping $\sigma: E\rightarrow\mathcal{S}'$ s.t.
$[\sigma (u)]_\infty = u$ for every $u\in E$.
\end{defn}
 In case $E$ is invariant by translations (i.e. if $\tau_a(E)\subset E$ for all $a\in \R$), we say that a realization $\sigma$ of $E$ {\em commutes with translations}  if $\sigma \circ \tau_a= \tau_a \circ \sigma$ for all $a\in \R$ ; such property holds iff  
 $\sigma (E)$ is invariant by translations.  Similar considerations apply to dilation invariance.
 \begin{prop} \label{transl} Let $E$ be a quasi-Banach space continuously embedded into $\mathcal{S}'_\infty$.
 Assume that $E$ admits a realization $\sigma$ s.t.
 \begin{equation} \label{inl1loc} \mathcal{F}(\sigma (E)) \subset L_{1,loc}(\euclid)\,.\end{equation} 
 Then
 \begin{equation} \label{sigma} \sigma(E)= \{ f\in \mathcal{S}'\,:\, [f]_\infty \in E \quad \mathrm{and}\quad \widehat{f}\in L_{1,loc}(\euclid)\}\,.\end{equation}
 \end{prop}
 \begin{proof} Let $V$ denote the r.h.s of (\ref{sigma}). Then $\sigma(E)\subset V$ follows immediately by (\ref{inl1loc}). In the opposite sense, if we take
$f\in V$, then $f- \sigma( [f]_\infty)\in {\mathcal P}_\infty$, thus $\mathcal{F}\left(f- \sigma( [f]_\infty)\right)$ is an integrable function near $0$,  supported by $\{0\}$. One concludes that $ f=\sigma( [f]_\infty) $, i.e. $f\in \sigma(E)$.
\end{proof}
\begin{cor} \label{uniquel1}Let $E$ be a quasi-Banach space continuously embedded into $\mathcal{S}'_\infty$.
Then $E$ admits {\em at most one} realization $\sigma$ s.t. {\em (\ref{inl1loc})}.
\end{cor}
\begin{prop}\label{inv} Let $E$ be a quasi-Banach space continuously embedded into $\mathcal{S}'_\infty$.
Assume that  $E$ is invariant by translations (resp. dilations).  If $E$ admits a realization $\sigma$ s.t. {\em (\ref{inl1loc})}, then
$\sigma$ commutes with translations (resp. dilations).
\end {prop}
\begin{proof} By (\ref{sigma}), $\sigma(E)$ is invariant by translations (resp. dilations).
\end{proof}
\begin{prop}\label{FL1loc}
 Let $E$ be a quasi-Banach space, continuously embedded into $\mathcal{S}'_\infty$, s.t. 
  \[\dot{{\mathcal F}} (E) \subset L_{1,loc}(\R\setminus \{0\})\,.\]
  Then the following properties are equivalent :
 \begin{itemize}
\item[(i)] for every $R>0$, there exists $c_R>0$ s.t., for all $u\in E$,
\begin{equation}\label{L1loc} \int_{ 0<|\xi|\leq R} \left| \dot{{\mathcal F}}(u) (\xi)\right|  \mathrm{d}\xi \leq c_R \|u\|_E\,,\end{equation}
\item[(ii)] 
there exists a realization $\sigma$ of $E$ s.t.
 $\mathcal{F}(\sigma (E)) \subset L_{1,loc}(\euclid)$.
 \end{itemize}
\end{prop}
\begin{proof} 
{\em Step 1 : (i) $\Rightarrow$ (ii).} Let $u\in E$ : by extending arbitrarily $\dot{\mathcal F}(u)$ at $0$, we obtain a function, say $T(u)$, 
in $L_{1,loc}(\R)\cap \mathcal{S}'$. By property (\ref{L1loc}), the mapping $T$, defined in such a way, is linear and continuous from $E$ to $\mathcal{S}'$. By setting $\sigma:= \mathcal{F}^{-1} \circ T$ we obtain a realization of $E$ s.t. $\mathcal{F}(\sigma ( E))\subset L_{1,loc}(\euclid)$.

{\em Step 2 : (ii) $\Rightarrow$ (i).} Let us endow $\sigma (E)$ with the norm $\| -\|_E$. By the Closed Graph theorem, the mapping $\mathcal{F}: \sigma (E) \rightarrow  L_{1,loc}(\euclid)$ is continuous. The estimate (\ref{L1loc}) follows.
\end{proof}

\subsection{The strong Szasz property}

By Proposition \ref{inv}, the strong Szasz property is related to the existence of translation commuting realizations. Let us recall the following statement :

\begin{prop}\label{realtrans} The space   $\dot{B}^s_{r,q}(\R)$ admits a translation commuting realization iff
\begin{equation}\label{condrB} s<n/r \quad \mathrm{or} \quad (\,s=n/r \quad \mathrm{and} \quad q\leq 1\,)\,.\end{equation}
The space   $\dot{F}^s_{r,q}(\R)$ admits a translation commuting realization iff
\begin{equation}\label{condrL} s<n/r \quad \mathrm{or} \quad (\,s=n/r \quad \mathrm{and} \quad r\leq 1\,)\,.\end{equation}
\end{prop}
\begin{proof} {\em Step 1.} Assume that  the conditions (\ref{condrB}) and (\ref{condrL}) hold, respectively. For any $u\in  \dot{A}^s_{r,q}(\R)$, the series
$ \sum_{j\in \z} Q_ju$ converges as well in $\temp{}$, see  \cite[prop.~4.6]{Bou_13} in case $\min (r,q)\geq 1$ and \cite[thm.~4.1]{Mou_15} in the general case ;
denoting by $\sigma_0 (u)$ its sum in $\temp{}$, we obtain a realization $\sigma_0 : \dot{A}^s_{r,q}(\R) \rightarrow \temp{}$, which clearly commutes with translations (Indeed $\sigma_0$ commutes as well with dilations, see  \cite[prop.~4.6 (2)]{Bou_13} and \cite[thm.~1.2]{Mou_15}). 

{\em Step 2.} Assume that the conditions (\ref{condrB}) and (\ref{condrL}) does not hold, respectively.
In case $\min (r,q)\geq 1$, we proved
 that, for some integer $\nu\geq 1$,
 $ \dot{A}^s_{r,q}(\R)$ admits no translation commuting realization in $\temp{\nu -1}$, the space of tempered distributions modulo polynomials of degree less than $\nu -1$, see  \cite[thm.~4.2 (2)]{Bou_13}. It is not difficult to see that this proof works as well for any $r,q>0$.
 {\it A fortiori} $ \dot{A}^s_{r,q}(\R)$ admits no translation commuting realization in $\temp{}$.
\end{proof}
\begin{thm}\label{strongSz}
 Let $s,p,q,r$ s.t. {\em (\ref{param})}. Then the system $(s,p,q,r,n)$ satisfies the property (s-SB) (resp. (s-SF)) iff both  conditions {\em (\ref{condB})} and  {\em (\ref{condrB})} 
(resp.  conditions {\em (\ref{condL})} and  {\em(\ref{condrL})}\,) hold.
\end{thm}
\begin{proof} The necessity of the various conditions follows by Propositions \ref{inv} and \ref{realtrans}, and Theorem \ref{thszasz}.
We turn now to the proof of {\em sufficiency}. According to Remark \ref{wtos},  it will suffice to find a realization $\sigma$ of  $\dot{A}^s_{r,q}(\R)$ s.t. 
(\ref{realinl1}). We limit ourselves to Besov spaces. Similar arguments work in case of  Lizorkin-Triebel spaces.

 {\em Step 1.}  Assume $0<r\leq 2$, $s<n/r$ and $q\leq r'$. According to Theorem \ref{thszasz}, the system $(s,r',q,r,n)$ has the property
 (w-SB) with Szasz exponent 
 \[ \theta = s+n-\frac{n}{r'}-\frac{n}{r} < n\left(1-\frac{1}{r'}\right)\,.\]
By H\"older's inequality, it holds
\[ \int_{ 0<|\xi|\leq R} \left| f(\xi)\right|  \mathrm{d}\xi \leq c\, R^{ n\left(1-\frac{1}{r'}\right)-\theta}\,
 \left(       \int_{\R\setminus \{0\}}     |\xi|^{\theta r'} |f(\xi)|^{r'}\,\mathrm{d}\xi                            \right)^{1/r'} \]
 for every $R>0$ and every $f\in L_{1,loc}(    \R\setminus \{0\})$. Then we conclude with the aid of 
 Proposition \ref{FL1loc}.
   
  {\em Step 2.}  Assume $0<r\leq 2$, $s=n/r$, $q\leq 1$ . The system $(s,1,q,r,n)$ has the property
 (w-SB) with Szasz exponent equal to $ 0$. That means :
  \[       \int_{\R \setminus \{0\}}     |\dot{\mathcal F}(u)(\xi)|\,\mathrm{d}\xi                                
   \leq c\, \|u\|_{\dot B^{s}_{r,q}}\, .          \]
We can apply again Proposition \ref{FL1loc}. \end{proof}

  \begin{rem} By known results on unicity of translation or dilation commuting realizations, see \cite[thms.~4.1, 4.2]{Bou_13}, the realization obtained in Theorem \ref{strongSz} coincide with the ``standard'' realization $\sigma_0$ introduced in the proof of Proposition \ref{realtrans}.
  \end{rem}
 
 \subsection*{Acknowledgment} I thank Madani Moussai, Herv\'e Qu\'effelec and Winfried Sickel for useful discussions in the preparation of the paper.

Universit\'e de Paris, I.M.J. - P.R.G (UMR 7586)\\
B\^atiment Sophie Germain\\
       Case 7012\\  
   F 75205 Paris Cedex 13   
                        
                        bourdaud@math.univ-paris-diderot.fr

% ------------------------------------------------------------------------
\end{document}